\documentclass[journal]{IEEEtran}\usepackage{amsmath}

\usepackage{amsfonts}

\newtheorem{theorem}{Theorem}

\newtheorem{definition}[theorem]{Definition}
\newtheorem{example}[theorem]{Example}

\newtheorem{proposition}[theorem]{Proposition}

\input{tcilatex}

\begin{document}

\title{A note on combining chaotic dynamical systems using the fuzzy logic
XOR operator.}
\author{Rezki Chemlal. \and Laboratoire de Math\'{e}matiques Appliqu\'{e}es,
Facult\'{e} des sciences exactes.Universit\'{e} Abderahmane Mira
Bejaia.06000 Bejaia Algeria.}
\maketitle

\begin{abstract}
In this paper we explore whatever combining two chaotic dynamical systems
using the fuzzy logic operator XOR\ can maintain or not the chaotic
properties of the resulting dynamical system. This study is motivated by
techniques used in applications to secure communications ,images encryption
and cryptography.

\end{abstract}

\section{Introduction}

Chaotic dynamical systems are commonly used as models in a wide range of
applications, cryptography \cite{Koc01},\cite{Bap98},\cite{Sch01},\cite%
{Vai03} image encryption and retrieval \cite{Mykola2020},\cite{Rou00},\cite%
{CD04} and achievement of associative memory properties \cite{Dmi02},\cite%
{Dmi91},\cite{Kwan03}.

Extreme sensitivity to initial condition is an interesting property of
chaotic systems. This property makes chaotic systems a worthy choice for
constructing cryptosystems or for image encryption.

Another idea is to mix two or more chaotic dynamical systems to gain more
"unpredictably" or more "confusion" in order to enhance encryption process 
\cite{Pareek05},\cite{Mykola2020}.

This rise a natural question, whatever combining two chaotic dynamical
systems permits to maintain chaotic property of the resulting one ? In
particular whatever combining two chaotic dynamical systems by fuzzy logic
operators, mainly the operator xor, gives rise to a chaotic dynamical system.

In other words consider two chaotic dynamical systems $\left( I,F\right) $
and $\left( I,G\right) $ is the dynamical system $\left( I,F\ast G\right) $
still chaotic ? Where * stands for any fuzzy logic operator.

We studied some combination of known chaotic dynamical systems and checked
whatever the combination is still chaotic or not. This gave us some
preliminary remarks about how to combine chaotic dynamical systems in order
to maintain the chaotic properties of the resulting dynamical system.

\section{Dynamical systems}

\subsection{Topological dynamics}

A dynamical system $\left( X,F\right) $ consists of a compact metric space $%
X $ and a continuous self--map $F$ .

A point $x$ is said periodic if there exists $p>0$ with $F^{p}\left(
x\right) =x.$ The least $p$ with this property is called the period of $x.$
A\ point $x$ is eventually or ultimately periodic if $T^{m}\left( x\right) $
is periodic for some $m\geq 0.$

In the same way a dynamical system is said periodic if there exists $p>0$
with $F^{p}\left( x\right) =x$ for every $x\in X$ and eventually periodic if 
$F^{m}$ is periodic for some $m\geq 0.$

A point $x\in X$ is said to be an equicontinuity point, or to be Lyapunov
stable, if for any $\epsilon >0$, there exists $\delta >0$ such that if $%
d\left( x,y\right) <\delta $ one has $d\left( F^{n}\left( y\right)
,F^{n}\left( x\right) \right) <\epsilon $ for any integer $n\geq 0$.

We say that $\left( X,F\right) $ is sensitive if for any $x\in X$ we have :%
\[
\exists \epsilon >0,\forall \delta >0,\exists y\in B_{\delta }\left(
x\right) ,\exists n\geq 0\text{ }\mathrm{such}\text{ \textrm{that }}d\left(
F^{n}\left( y\right) ,F^{n}\left( x\right) \right) \geq \epsilon .
\]

We say that $\left( X,T\right) $ is expansive if we have :%
\[
\exists \epsilon >0,\forall x\neq y\in X,\exists n\geq 0,d\left( F^{n}\left(
x\right) ,F^{n}\left( y\right) \right) \geq \epsilon .
\]

A dynamical system $\left( X,T\right) $\emph{\ }is transitive if for any
nonempty open sets $U,V\subset A^{\mathbb{Z}}$ there exists $n>0$ with $%
U\cap F^{-n}\left( V\right) \neq \emptyset .$ This is equivalent to the
existence of a point with a dense orbit.

A dynamical system is said topologically mixing if for any nonempty open
sets $U,V\subset A^{\mathbb{Z}},U\cap F^{-n}\left( V\right) \neq \emptyset $
for all sufficiently large $n$.

Let $\left( X,F\right) $ be a dynamical system, endow the set $X$ with a
sigma algebra $\mathbb{B}$. The function $F$ preservers some measure $\mu $
on the sigma algebra $\mathbb{B}$ iff for every $B\in \mathbb{B}$ we have $%
\mu (F^{-1}(B))=\mu (B).$ We say then that $\left( X,\mathbb{B},F,\mu
\right) $ is a measurable dynamical system.

The topological support of a measure is defined as the set of all points $%
x\in X$ for which every open neighborhood of $x$ has positive measure.

A dynamical system $\left( X,\mathbb{B},F,\mu \right) $ is ergodic if every
invariant subset of $X$ is either of measure 0 or of measure 1.
Equivalently, if for any measurable $U,V\subset X,$ there exists some $n\in 
\mathbb{N}$ such that $\mu \left( U\cap F^{-n}\left( V\right) \right) >0.$

If a dynamical system is ergodic then it is transitive on the topological
support of the measure.

For dynamical systems on the real line a common used measure is the Lebesgue
measure, the topological support of the Lebesgue measure is $\mathbb{R}$%
.\newpage

\subsection{Chaotic dynamical systems on the interval}

Dynamical systems defined on the real line have a particular behavior, a
rich literature is devoted for the subject, we will recall here some results
about their properties.

In matter of chaos The Devaney's chaos is seen as a combination of
unpredictably (sensitivity) and regular behaviors (periodic points),
transitivity ensuring that the system is undecomposable.

\begin{definition}
A topological dynamical system $\left( X,f\right) $ is chaotic in the sense
of Devaney if :\newline
(1) Is transitive.\newline
(2) The set of periodic points is dense in X.\newline
(3) Is sensitive to initial conditions.
\end{definition}

It is know that for every dynamical system the conditions 1 and 2 implies
the condition 3 which lead to the so called Modified Devaney definition of
chaos.

For interval maps transitivity is enough to imply the other two conditions.

\begin{proposition}
An interval map is chaotic in the sense of Devaney if and only if it is
transitive.
\end{proposition}

For the proof of this result you can look at \cite{Vel94}.

\section{Results}

\subsection{Preliminary remarks}

One issue when using the fuzzy logic xor\ operator is preserving the
invariance of the resulting dynamical system if we want to combine two
dynamical systems they must be defined on the same interval but this is not
enough to ensure invariance of the resulting combined dynamical system.

If the two dynamical systems are defined on the interval $\left[ 0,1\right] $
it is easy to show that the result will be invariant on the interval $\left[
0,1\right] .$ One solution to overcome the problem of the invariance is to
rescale every dynamical system defined on a given interval to $\left[ 0,1%
\right] .$

Below we give two examples the first one of an xor combination of two
chaotic dynamical systems which is not chaotic and the second one where the
result is a chaotic dynamical system.

\begin{example}
Let us consider the two dynamical systems $\left( \left[ 0,1\right]
,f_{r}\right) $ and $\left( \left[ 0,1\right] ,T\right) $ where $f_{r}$ and $%
T$ are the logistic map and the tent map respectively. 
\[
f_{r}\left( x\right) =r.x.\left( 1-x\right) ,T\left( x\right) =\left\{ 
\begin{array}{c}
-2x+1,0\leq x\leq 0.5 \\ 
2x-1,0.5\leq x\leq 1%
\end{array}%
\right.
\]%
These two dynamical systems are well known chaotic systems , let us consider
their fuzzy xor combination H defined by%
\[
H\left( x\right) =\max \left( f_{r}\left( x\right) ,T\left( x\right) \right)
-\min \left( f_{r}\left( x\right) ,T\left( x\right) \right)
\]%
It possesses two fixed points, the fixed point 0 is instable while the fixed
point 0.23 is asymptotically stable. The basin of attraction of the point $0$
contains 4 isolated points while the basin of attraction of 0.23 contain the
hole interval except the basin of attraction of 0.
\end{example}

\begin{example}
Consider the two following dynamical systems the two dynamical systems $%
\left( \left[ 0,1\right] ,T\right) $ and $\left( \left[ 0,1\right]
,ST\right) $ where $T$ is the tent map and ST that is defined by 
\[
ST\left( x\right) =\left\{ 
\begin{array}{c}
-2x+1:0\leq x\leq \frac{1}{2} \\ 
2x-1:\frac{1}{2}\leq x\leq \frac{1}{4}%
\end{array}%
\right. 
\]%
The graph of the function $ST$ is given below, you can see it as an inverted
Tent map graph. The map $ST$ is chaotic as the point $\frac{\pi }{3.5}$ has
a dense orbit. The map $S$ $xor$ $ST$ is chaotic .
\end{example}

\subsection{Numerical experiments summary}

Along with the two examples shown before we have tested some other dynamical
systems, part of the results is shown in the following table.%
\[
\begin{tabular}{|l|l|l|l|l|l|}
\hline
{\tiny xor} & {\tiny Doubling\ map} & {\tiny Cubic\ map} & {\tiny Logistic\
map} & {\tiny Tent\ map} & {\tiny Inverted\ Tent\ map} \\ \hline
{\tiny Doubling\ map} & {\tiny .....} & {\tiny Non chaotic} & {\tiny Non
chaotic} & {\tiny Non chaotic} & {\tiny Non chaotic} \\ \hline
{\tiny Cubic\ map} & {\tiny .....} & {\tiny .....} & {\tiny Non chaotic} & 
{\tiny Non chaotic} & {\tiny Non chaotic} \\ \hline
{\tiny Logistic\ map} & {\tiny .....} & {\tiny .....} & {\tiny .....} & 
{\tiny Non chaotic} & {\tiny Chaotic} \\ \hline
{\tiny Tent\ map} & {\tiny .....} & {\tiny .....} & {\tiny .....} & {\tiny %
.....} & {\tiny Chaotic} \\ \hline
{\tiny Inverted\ Tent\ map} & {\tiny .....} & {\tiny .....} & {\tiny .....}
& {\tiny .....} & {\tiny .....} \\ \hline
\end{tabular}%
\]

The observation of the results of the table suggests that combining two
dynamical systems using the xor operator leads to the resulting dynamical
system to be chaotic if their graphs have some form of symmetry to the
horizontal line $y=\frac{1}{2}.$

\subsection{Mirror effect and number of full branches}

The aime of this section is to come with some criterion choice. The optimal
situation is that the two dynamical systems have to be symmetrical to the
horizontal line $y=\frac{1}{2}$ this is what we will call a mirror effect.
In this situation we can show that the combination is a chaotic dynamical
using tools from ergodic theory.

In practical this result could be a tool to choose the dynamical systems to
combine, the closet to symmetry to $y=\frac{1}{2}$ they are the best chances
the combination to work we have.

\begin{definition}
Let $I\subset \mathbb{R}$ be an interval. A map $f:I\rightarrow I$ is a full
branch map if there exists a finite or countable partition $\mathcal{P}$ of $%
I$ into subintervals such that for each $w\in \mathcal{P}$ the map $\left.
f\right\vert _{int\left( w\right) }:int\left( w\right) \rightarrow int\left(
I\right) $ is a bijection.\newline
A map $f$ is a piecewise continuous (resp $C^{1},C^{2},affine$) full branch
map if for each $w\in \mathcal{P}$ the map $\left\vert f\right. _{int\left(
w\right) }:int\left( w\right) \rightarrow int\left( I\right) $ is a
homeomorphism (resp $C^{1}$ diffeomorphism$,C^{2}$ diffeomorphism$,affine$).
\end{definition}

\begin{definition}
A full branch map has bounded distortion if 
\[
\underset{n\in \left\{ 1,2\right\} }{\sup }\underset{w^{\left( n\right) }\in 
\mathcal{P}^{\left( n\right) }}{\sup }\underset{x,y\in w^{\left( n\right) }}{%
\sup }\log \left\vert Df^{\left( n\right) }\left( x\right) /Df^{\left(
n\right) }\left( y\right) \right\vert <\infty
\]%
\newline
here $\left( n\right) $ stands for the $nth$ derivative.
\end{definition}

If the function is piecewise affine then the distortion is 0.

\begin{example}
The tent map, the doubling map and the logistic map have two full branches
with bounded distortion, the cubic map have three full branches with bounded
distortion.
\end{example}

\begin{proposition}
Consider $\left( \left[ 0,1\right] ,f\right) $ and $\left( \left[ 0,1\right]
,g\right) $ two dynamical systems, suppose that the graphs of $f$ and $g$
are symmetrical according to the horizontal line $y=\frac{1}{2}$ and that
the number of full branches of $f$ and $g$ are equal to $k$.\newline
The dynamical system $\left( \left[ 0,1\right] ,f\text{ }xor\text{ }g\right) 
$ has $2k$ full branches.
\end{proposition}

\begin{proof}
Suppose that we have a partition $\mathcal{P}$ of $\left[ 0,1\right] $ into
subintervals such that for each $w\in \mathcal{P}$ the map $\left.
f\right\vert _{int\left( w\right) }:int\left( w\right) \rightarrow \left] 0,1%
\right[ $ is a bijection.\newline
As $g$ is a mirror of $f$ we obtain by symetry 
\[
\left( f\text{ }xor\text{ }g\right) \left( x\right) =\left\{ 
\begin{array}{c}
2\left\vert f\left( x\right) -0.5\right\vert \text{ }if\text{ }f\left(
x\right) \leq 0.5 \\ 
2\left\vert f\left( x\right) -0.5\right\vert \text{ }if\text{ }f\left(
x\right) \geq 0.5%
\end{array}%
\right. 
\]%
As $f$ is a bijection there is a partition $w=w_{1}\cup w_{2}$ such that : 
\begin{eqnarray*}
\left( f\text{ }xor\text{ }g\right) \left( x\right)  &=&\left\{ 
\begin{array}{c}
2\left\vert f\left( x\right) -0.5\right\vert \text{ }if\text{ }x\in w_{1} \\ 
2\left\vert f\left( x\right) -0.5\right\vert \text{ }if\text{ }x\in w_{2}%
\end{array}%
\right.  \\
&\Rightarrow &\left\{ 
\begin{array}{c}
\left( f\text{ }xor\text{ }g\right) \left( w_{1}\right) =\left] 0,1\right] 
\\ 
\left( f\text{ }xor\text{ }g\right) \left( w_{1}\right) =\left] 0,1\right[ 
\end{array}%
\right. 
\end{eqnarray*}%
Thus $f$ $xor$ $g$ has two full branches on $w.$
\end{proof}

\begin{proposition}
Consider $\left( \left[ 0,1\right] ,f\right) $ and $\left( \left[ 0,1\right]
,g\right) $ two dynamical systems, suppose that the graphs of $f$ and $g$
are symmetrical according to the horizontal line $y=\frac{1}{2}$ and that $f$
and $g$ have the full branch property then $f$ $xor$ $g$ is chaotic .
\end{proposition}

\begin{proof}
Suppose that $f$ has $k$ branches then $f$ $xor$ $g$ has $2k$ branches and a
relevant partition $w.$\newline
As the graph of $g$ is symmetrical to the graph of $f$ then they have the
same distortion.

On each branch of the partition we have 
\begin{multline*}
\underset{n\geq 1}{\sup }\underset{w^{\left( n\right) }\in \mathcal{P}%
^{\left( n\right) }}{\sup }\underset{x,y\in w^{\left( n\right) }}{\sup }\log
\left\vert D\left( f\text{ }xor\text{ }g\right) ^{n}\left( x\right) /D\left(
f\text{ }xor\text{ }g\right) ^{n}\left( y\right) \right\vert \leq  \\
\max \left( 
\begin{array}{c}
\underset{n\geq 1}{\sup }\underset{w^{\left( n\right) }\in \mathcal{P}%
^{\left( n\right) }}{\sup }\underset{x,y\in w^{\left( n\right) }}{\sup }\log
\left\vert Df\text{ }^{\left( n\right) }\left( x\right) /D\left( f\text{ }%
\right) ^{n}\left( y\right) \right\vert  \\ 
,\underset{n\geq 1}{\sup }\underset{w^{\left( n\right) }\in \mathcal{P}%
^{\left( n\right) }}{\sup }\underset{x,y\in w^{\left( n\right) }}{\sup }\log
\left\vert D\left( g\right) ^{n}\left( x\right) /D\left( g\right) ^{n}\left(
y\right) \right\vert 
\end{array}%
\right) 
\end{multline*}%
Hence $f$ $xor$ $g$ is a full branch map with bounded distortion. Then the
Lebesgue measure is ergodic \cite{Ste}.\newline
As the Lebesgue measure is ergodic then $f$ $xor$ $g$ is transitive on the
topological support of the Lebesgue measure.Hence it is chaotic.
\end{proof}

\section{Conclusion}

We investigated an ideal situation to combine chaotic one dimensional maps
using the fuzzy logic xor operator.

Using tools from ergodic theory we were able to establish a result using
slightly strong condition than Devaney's chaos, it is worth noting that this
condition is satisfied by a majority of classical chaotic maps on the
interval.

\end{document}